\documentclass[11pt,a4paper]{article}
\usepackage[a4paper]{geometry}
\usepackage{amssymb,latexsym,amsmath,amsfonts,amsthm}
\usepackage{graphicx}
\usepackage{algorithm}
\usepackage{algorithmic}
 \usepackage{dsfont}
\usepackage{mathrsfs}
\usepackage{epsfig}
\usepackage{booktabs}
\usepackage{mathtools}
\usepackage{comment,verbatim}
\usepackage{overpic}
\usepackage{hyperref}
\usepackage{bbm}
\usepackage[toc,page]{appendix}

\newtheorem{theorem}{Theorem}[section]
\newtheorem{lemma}[theorem]{Lemma}

\newtheorem{corollary}[theorem]{Corollary}

\theoremstyle{definition}

\theoremstyle{definition}

\theoremstyle{remark}

\newtheorem{remark}[theorem]{Remark}

\numberwithin{equation}{section}

\usepackage{color}

\usepackage[normalem]{ulem}

\def\XXint#1#2#3{{
\setbox0=\hbox{$#1{#2#3}{\int}$}
\vcenter{\hbox{$#2#3$}}\kern-.5\wd0}}

\begin{document}
\title{A new and sharper bound for Legendre expansion of differentiable functions} 
\author{Haiyong Wang\footnotemark[2]~\footnotemark[3]
}

\maketitle
\renewcommand{\thefootnote}{\fnsymbol{footnote}}

\footnotetext[2]{School of Mathematics and Statistics, Huazhong
University of Science and Technology, Wuhan 430074, P. R. China.
E-mail: \texttt{haiyongwang@hust.edu.cn}}

\footnotetext[3]{Hubei Key Laboratory of Engineering Modeling and
Scientific Computing, Huazhong University of Science and Technology,
Wuhan 430074, China.}

\begin{abstract}
In this paper, we provide a new and sharper bound for the Legendre
coefficients of differentiable functions and then derive a new error
bound of the truncated Legendre series in the uniform norm. The key
idea of proof relies on integration by parts and a sharp
Bernstein-type inequality for the Legendre polynomial. An
illustrative example is provided to demonstrate the sharpness of our
new results.
\end{abstract}


\vspace{0.05in}


\section{Introduction}\label{sec:introduction}
Let $P_n(x)$ be the Legendre polynomial of degree $n$ which is
defined by
\begin{align}
P_n(x) = \frac{1}{2^n n!} \frac{d^n}{dx^n} \left[ (x^2-1)^n \right],
\quad n\geq0.
\end{align}
The set of Legendre polynomials $\{P_0(x),P_1(x),\cdots\}$ form a
system of polynomials orthogonal on the interval $[-1,1]$ with
respect to the weight function $\omega(x)=1$ and
\begin{align}\label{def:LegendrePoly}
\int_{-1}^{1} P_n(x) P_m(x)dx = h_n \delta_{mn},
\end{align}
where $\delta_{mn}$ is the Kronecker delta and
\begin{align}\label{eq:constant}
h_n = \left(n+\frac{1}{2} \right)^{-1}.
\end{align}
The Legendre polynomials are widely used in many branches of
scientific computing such as interpolation and approximation, the
construction of quadrature formulas and spectral methods for
differential equations.

The Legendre expansion of a
function $f:=[-1,1]\rightarrow \mathbb{R}$ is defined by
\begin{align}\label{eq:LegExp}
f(x) = \sum_{n=0}^{\infty} a_n P_n(x),
\end{align}
where the Legendre coefficients are given by
\begin{align}
a_n = h_n^{-1} \int_{-1}^{1} f(x) P_n(x) dx.
\end{align}
The problem of estimating the magnitude of the Legendre coefficients
$a_n$ is of particular interest both from the theoretical and
numerical point of view. Indeed, it is useful not only in
understanding the rate of convergence of Legendre expansion but
useful also in estimating the degree of the Legendre polynomial
approximation to $f(x)$ within a given accuracy.

When $f(x)$ is analytic in a neighborhood of the interval $[-1,1]$,
we note that the estimate of the Legendre coefficients, or more
generally, the Gegenbauer and Jacobi coefficients, has been studied
in
\cite{wang2012legendre,wang2016gegenbauer,xiang2012error,zhao2013sharp}.
The analysis in those references is either built on the connection
relation between Chebyshev and Legendre polynomials
\cite{wang2012legendre,xiang2012error,zhao2013sharp} or built on the
contour integral expression of the Legendre coefficients
\cite{wang2016gegenbauer}.

In this work, we are interested in the case where $f(x)$ is a
differentiable function. We first establish a new bound for the
Legendre coefficients and the key ingredient here is that the
Legendre polynomial satisfies a sharp Bernstein-type inequality. An
illustrative example is provided to show that our new result is
sharper than the result given in \cite[Theorem
2.1]{wang2012legendre}. Furthermore, we then derive a new error
bound of Legendre expansion in the uniform norm. Our main results
are stated in the next section.

\section{A new and sharper bound for Legendre coefficients of differentiable functions}\label{sec:bound}
In this section we state an explicit and computable bound for the
Legendre coefficients of differentiable functions. This new bound,
as will be shown later, is sharper than the one given in
\cite[Theorem 2.1]{wang2012legendre}. Before proceeding, we first
define the weighted semi-norm
\begin{align}
\| f \| := \int_{-1}^{1} \frac{|f{'}(x)|}{(1-x^2)^{\frac{1}{4}}} dx.
\end{align}
The following Bernstein-type inequality of Legendre polynomials will
be useful.
\begin{lemma}\label{lem:LegPolyBound}
For $x\in[-1,1]$ and $n\geq0$, we have
\begin{align}\label{eq:LegPolyBound}
(1-x^2)^{\frac{1}{4}} |P_n(x)| < \sqrt{\frac{2}{\pi}}
\left(n+\frac{1}{2}\right)^{-\frac{1}{2}}.
\end{align}
Moreover, the above inequality is optimal in the sense that the
factor $(n+\frac{1}{2})^{-\frac{1}{2}}$ can not be improved to
$(n+\frac{1}{2}+\epsilon)^{-\frac{1}{2}}$ for any $\epsilon>0$ and
the constant $\sqrt{2/\pi}$ is best possible.
\end{lemma}
\begin{proof}
See \cite{Antonov1981estimate,lorch1983altern}.
\end{proof}

\begin{figure}[h]
\centering
\begin{overpic}
[width=4.8cm]{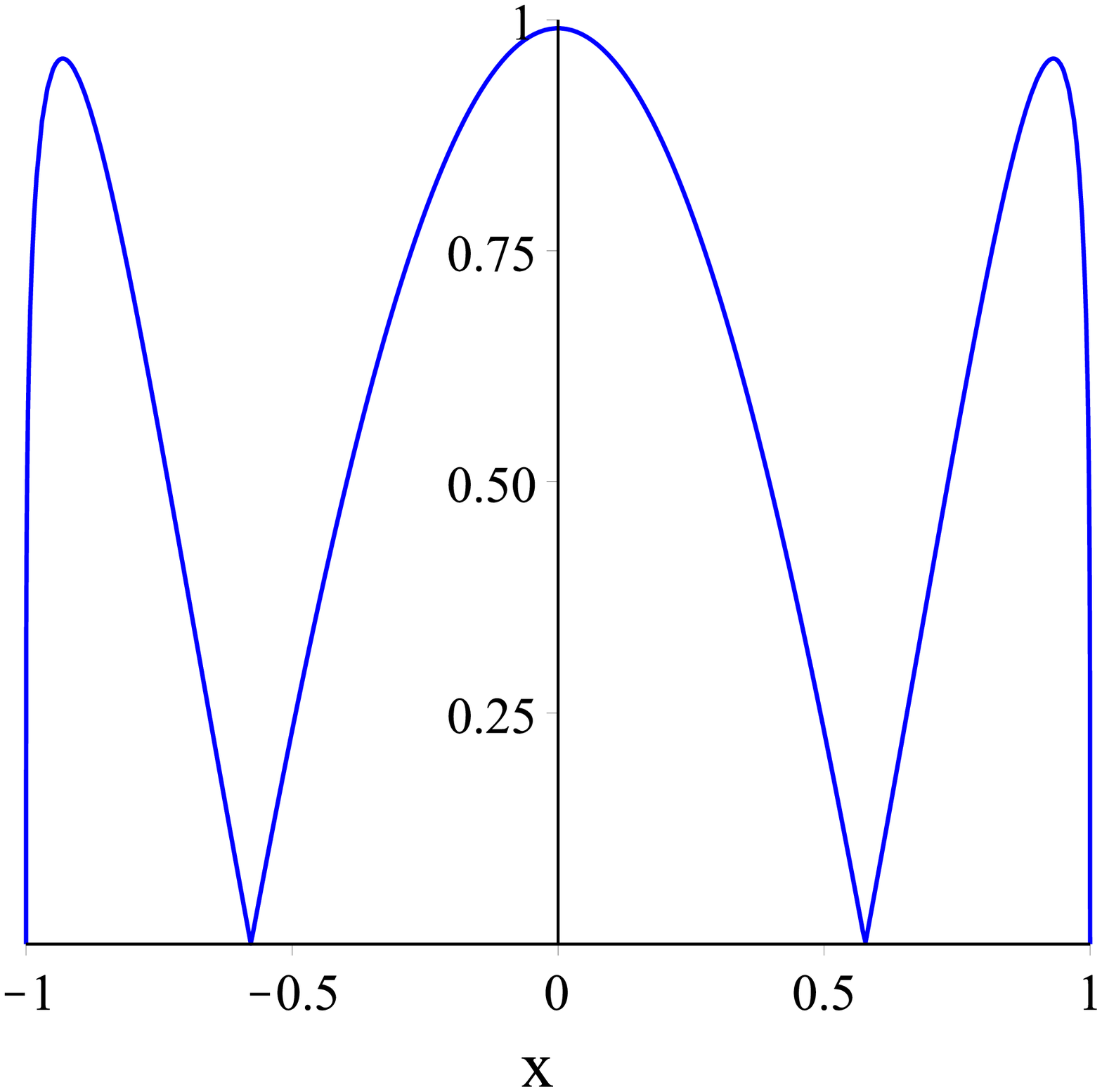}
\end{overpic}
\begin{overpic}
[width=4.8cm]{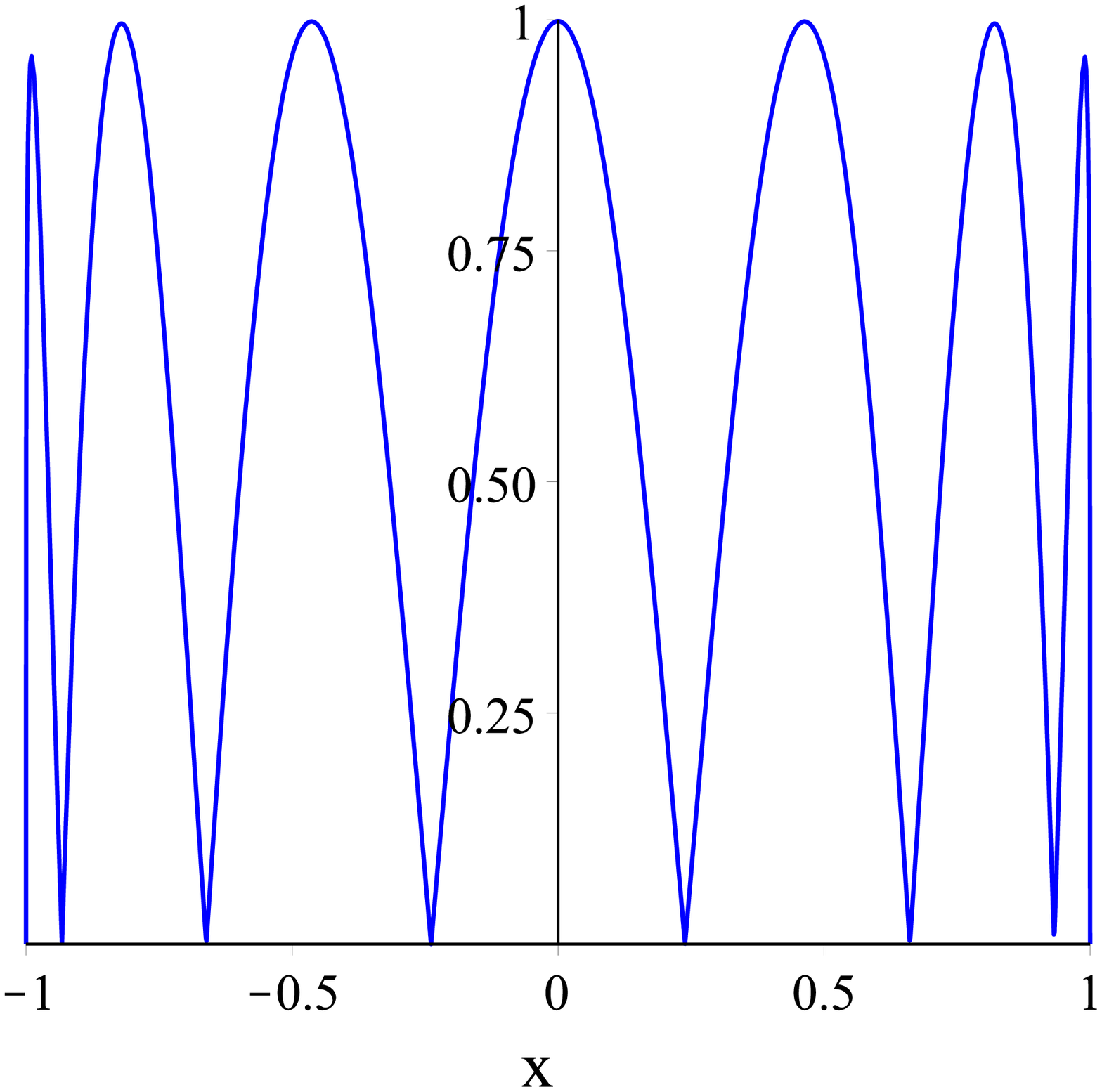}
\end{overpic}
\begin{overpic}
[width=4.8cm]{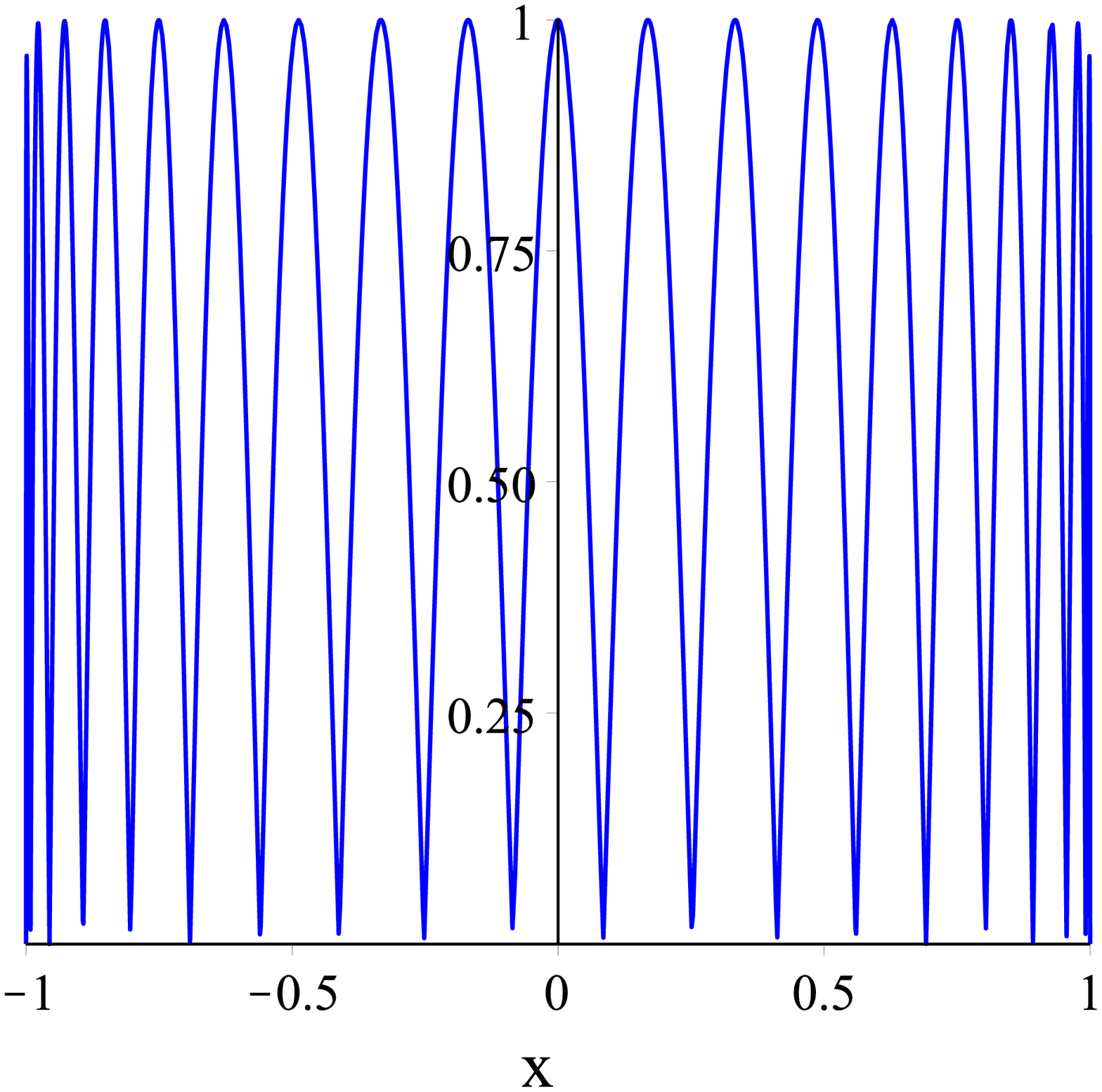}
\end{overpic}
\caption{The ratio of the term on the left-hand side to the term on
the right-and side of \eqref{eq:LegPolyBound} for $n=2$ (left),
$n=6$ (middle) and $n=18$ (right). } \label{fig:LegendreBound}
\end{figure}

To show the sharpness of \eqref{eq:LegPolyBound}, we consider the
ratio of the term on the left-hand side to the term on the
right-hand side as a function of $x$. Numerical results are
presented in Figure \ref{fig:LegendreBound} for three values of $n$.
It is clear to see that the maximum value of the ratio is very close
to one.

We are now ready to state our first main result on the bound of
Legendre coefficients for differentiable functions.
\begin{theorem}\label{thm:NewBound}
Assume that $f,f{'},\ldots,f^{(m-1)}$ are absolutely continuous and
the $m$th derivative $f^{(m)}(x)$ is of bounded variation.
Furthermore, assume that $V_m = \|f^{(m)}\|<\infty$. Then, for
$n\geq m+1$,
\begin{align}
|a_n| \leq \frac{2V_m}{\sqrt{\pi(2n-2m-1)}} \prod_{k=1}^{m} h_{n-k}.
\end{align}
where $h_n$ is defined as in \eqref{eq:constant} and the product is
assumed to be one when $m=0$.
\end{theorem}
\begin{proof}
The basic idea of our proof is to employ integration by parts and
the inequality in Lemma \ref{lem:LegPolyBound}. By combining
\cite[Equation 18.9.7]{olver2010nist} and \cite[Eqation
18.9.19]{olver2010nist}, we have
\begin{align}\label{eq:LegRec}
P_n(x) = \frac{P_{n+1}{'}(x) - P_{n-1}{'}(x)}{2h_n^{-1}}, \quad
n\geq1,
\end{align}
Substituting this into \eqref{eq:LegExp} and applying integration by
part once, we obtain
\begin{align}\label{eq:StepOne}
a_n &= h_n^{-1} \int_{-1}^{1} f(x) P_n(x) dx \nonumber \\
&= \int_{-1}^{1} f(x) \frac{P_{n+1}{'}(x) - P_{n-1}{'}(x)}{2} dx
\nonumber \\
&= \left[ f(x) \frac{P_{n+1}(x) - P_{n-1}(x)}{2} \right]_{-1}^{1} +
\int_{-1}^{1} f{'}(x) \frac{P_{n-1}(x) - P_{n+1}(x)}{2} dx.
\end{align}
Furthermore, by making use of $P_n(\pm1) = (\pm1)^n$ for each
$n\geq0$, it is easy to see that the first term in the last equation
vanishes and therefore
\begin{align}\label{eq:LegendreStepOne}
a_n &= \int_{-1}^{1} f{'}(x) \frac{P_{n-1}(x) - P_{n+1}(x)}{2} dx.
\end{align}
This together with the result of Lemma \ref{lem:LegPolyBound} gives
\begin{align}
|a_n| &\leq \int_{-1}^{1} \frac{|f{'}(x)|}{2} \left[
\frac{2(1-x^2)^{-\frac{1}{4}}}{\sqrt{\pi(2n-1)}} +
\frac{2(1-x^2)^{-\frac{1}{4}}}{\sqrt{\pi(2n+3)}} \right] dx
\nonumber \\
&\leq \frac{2}{\sqrt{\pi(2n-1)}} \int_{-1}^{1}
\frac{|f{'}(x)|}{(1-x^2)^{\frac{1}{4}}} dx \nonumber \\
&=
\frac{2V_0}{\sqrt{\pi(2n-1)}}. \nonumber
\end{align}
This proves the case $m=0$.

When $m=1$, integrating by part to \eqref{eq:LegendreStepOne} again,
we get
\begin{align}\label{eq:StepTwo}
a_n &= \int_{-1}^{1} \frac{f{'}(x)}{2} \left[ \frac{P_{n}{'}(x) -
P_{n-2}{'}(x)}{2h_{n-1}^{-1}} - \frac{P_{n+2}{'}(x) -
P_{n}{'}(x)}{2h_{n+1}^{-1}} \right] dx \nonumber \\
&= \left[ f{'}(x) \frac{P_n(x) - P_{n-2}(x)}{4h_{n-1}^{-1}}
\right]_{-1}^{1} - \left[ f{'}(x) \frac{P_{n+2}(x) -
P_{n}(x)}{4h_{n+1}^{-1}} \right]_{-1}^{1} \nonumber \\
&~~~~~~~~ + \int_{-1}^{1} f{''}(x) \left[
\frac{P_{n-2}(x)}{4h_{n-1}^{-1}} - \frac{P_{n}(x)}{4h_{n-1}^{-1}} -
\frac{P_{n}(x)}{4h_{n+1}^{-1}} + \frac{P_{n+2}(x)}{4h_{n+1}^{-1}}
\right]dx.
\end{align}
We see that the first two terms in the last equation vanish and
therefore
\begin{align}
a_n &= \int_{-1}^{1} f{''}(x) \left[
\frac{P_{n-2}(x)}{4h_{n-1}^{-1}} - \frac{P_{n}(x)}{4h_{n-1}^{-1}} -
\frac{P_{n}(x)}{4h_{n+1}^{-1}} + \frac{P_{n+2}(x)}{4h_{n+1}^{-1}}
\right]dx.
\end{align}
By using the inequality in Lemma \ref{lem:LegPolyBound} again, we
obtain
\begin{align}
|a_n| &\leq \int_{-1}^{1} \frac{|f{''}(x)|}{(1-x^2)^{\frac{1}{4}}}
\left[ \frac{h_{n-1}}{2 \sqrt{\pi(2n-3)}} +
\frac{h_{n-1}}{2\sqrt{\pi(2n+1)}} \right. \nonumber \\
&~~~~~~~~ \left. + \frac{h_{n+1}}{2\sqrt{\pi(2n+1)}} +
\frac{h_{n+1}}{2\sqrt{\pi(2n+5)}} \right]dx \nonumber \\
&\leq \frac{2h_{n-1}}{\sqrt{\pi(2n-3)}} \int_{-1}^{1}
\frac{|f{''}(x)|}{(1-x^2)^{\frac{1}{4}}} dx \nonumber \\
&= \frac{2V_1}{\sqrt{\pi(2n-3)}} h_{n-1},
\end{align}
where we have used the property that $h_n$ is strictly decreasing
with respect to $n$ in the second step and this proves the case
$m=1$.

When $m\geq2$, we may continue the above process and this brings in
higher derivatives of $f$ and corresponding higher variations up to
$V_m$. Hence we can obtain the desired result.
\end{proof}

How sharp is Theorem \ref{thm:NewBound}? We consider the following
example
\begin{align}\label{eq:testfunction}
f(x) = |x-t|,
\end{align}
where $t\in(-1,1)$. It is easy to see that this function has a jump
in the first order derivative at $x=t$. 
In this case, it is readily verified that $m=1$ and $V_m =
2(1-t^2)^{-\frac{1}{4}}$ and therefore the result of Theorem
\ref{thm:NewBound} can be written explicitly as
\begin{align}\label{eq:ExamBound1}
|a_n| \leq \frac{2V_1}{\sqrt{\pi(2n-3)}}
\left(n-\frac{1}{2}\right)^{-1} = \frac{4
(1-t^2)^{-\frac{1}{4}}}{\sqrt{\pi(2n-3)}}
\left(n-\frac{1}{2}\right)^{-1}.
\end{align}
Let $B_1(n)$ denote the bound on the right-hand side of
\eqref{eq:ExamBound1}. We compare $B_1(n)$ with the absolute values
of the Legendre coefficients $|a_n|$ and numerical results are
illustrated in Figure \ref{fig:Example} for two values of $t$. We
can see clearly that, in the case of $t=0$, i.e., $f(x) = |x|$, our
bound $B_1(n)$ is almost indistinguishable with $|a_n|$ as $n$
increases. In fact, when $n=400$, we have $B_1(n)\approx
0.0002000963242$ and $|a_{n}|\approx 0.0001991004306$. It is clear
to see that our bound is rather sharp.

In \cite[Theorem 2.1]{wang2012legendre}, an upper bound of the
Legendre coefficients was given by
\begin{align}\label{eq:Bound12}
|a_n| \leq
\frac{\widehat{V}_m}{(n-\frac{1}{2})(n-\frac{3}{2})\cdots(n-m+\frac{1}{2})}
\sqrt{\frac{\pi}{2(n-m-1)}},
\end{align}
where $n\geq m+2$ and
\[
\widehat{V}_m = \int_{-1}^{1}
\frac{|f^{(m+1)}(x)|}{(1-x^2)^{\frac{1}{2}}}dx.
\]
We now make a comparison between the result of Theorem
\ref{thm:NewBound} with the bound on the right-hand side of
\eqref{eq:Bound12}. For simplicity, we let $B_{2}(n)$ denote the
bound on the right-hand side of \eqref{eq:Bound12}. For the function
\eqref{eq:testfunction}, we have $\widehat{V}_1 =
2(1-t^2)^{-\frac{1}{2}}$ and thus
\begin{align}\label{eq:ExamBound2}
B_2(n) = \frac{\widehat{V}_1}{n-\frac{1}{2}}
\sqrt{\frac{\pi}{2(n-2)}} =
\frac{2(1-t^2)^{-\frac{1}{2}}}{n-\frac{1}{2}}
\sqrt{\frac{\pi}{2(n-2)}}.
\end{align}
Comparing $B_1(n)$ and $B_2(n)$, we have for $n\geq3$ that
\begin{align}
\frac{B_1(n)}{B_2(n)} = \frac{2(1-t^2)^{\frac{1}{4}}}{\pi}
\sqrt{\frac{2n-4}{2n-3}} < \frac{2(1-t^2)^{\frac{1}{4}}}{\pi}.
\nonumber
\end{align}
Clearly, we see that the new bound is always sharper; see Figure
\ref{fig:Example}.

\begin{figure}[h]
\centering
\begin{overpic}
[width=7.2cm]{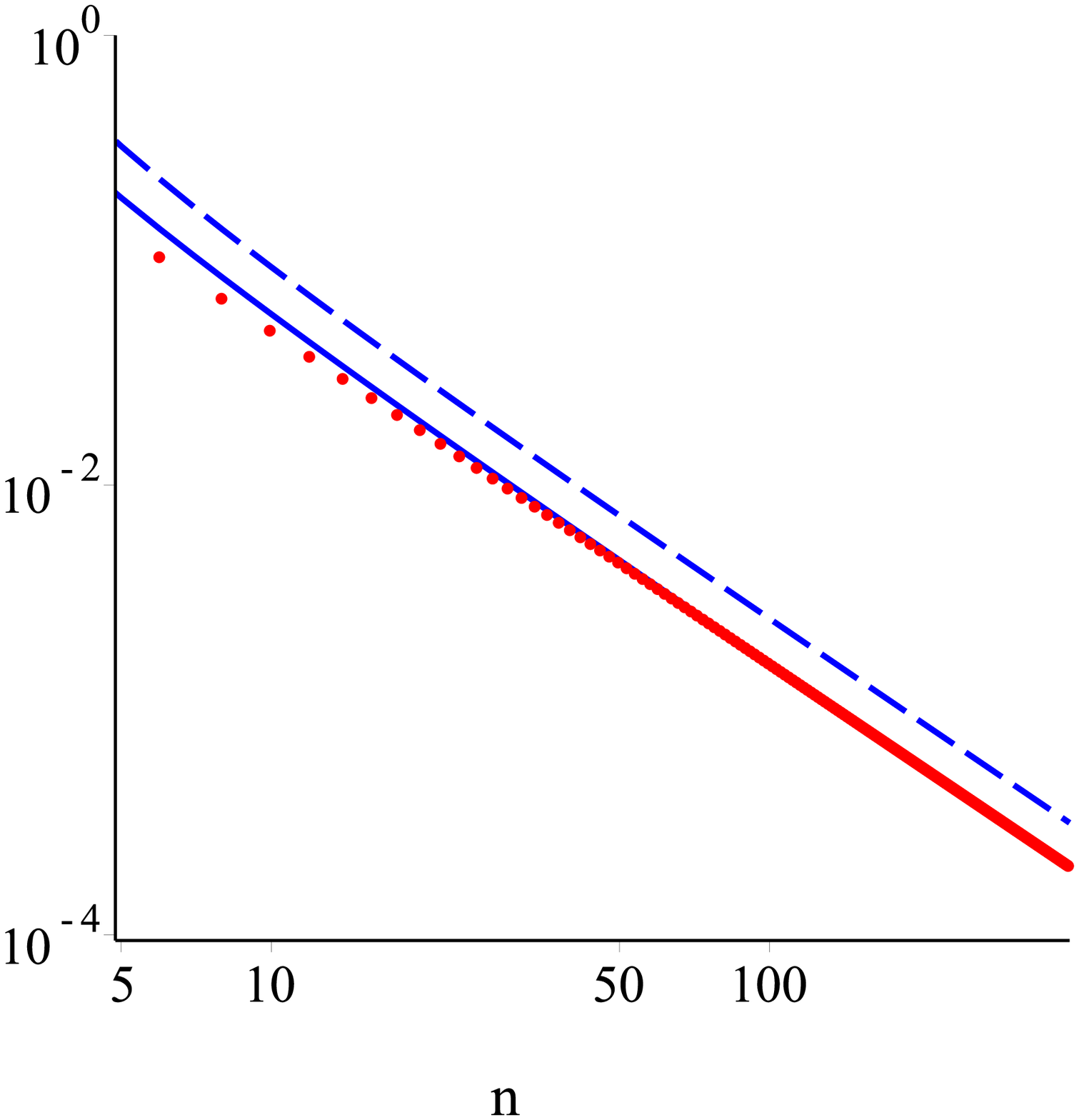}
\end{overpic}
\begin{overpic}
[width=7.2cm]{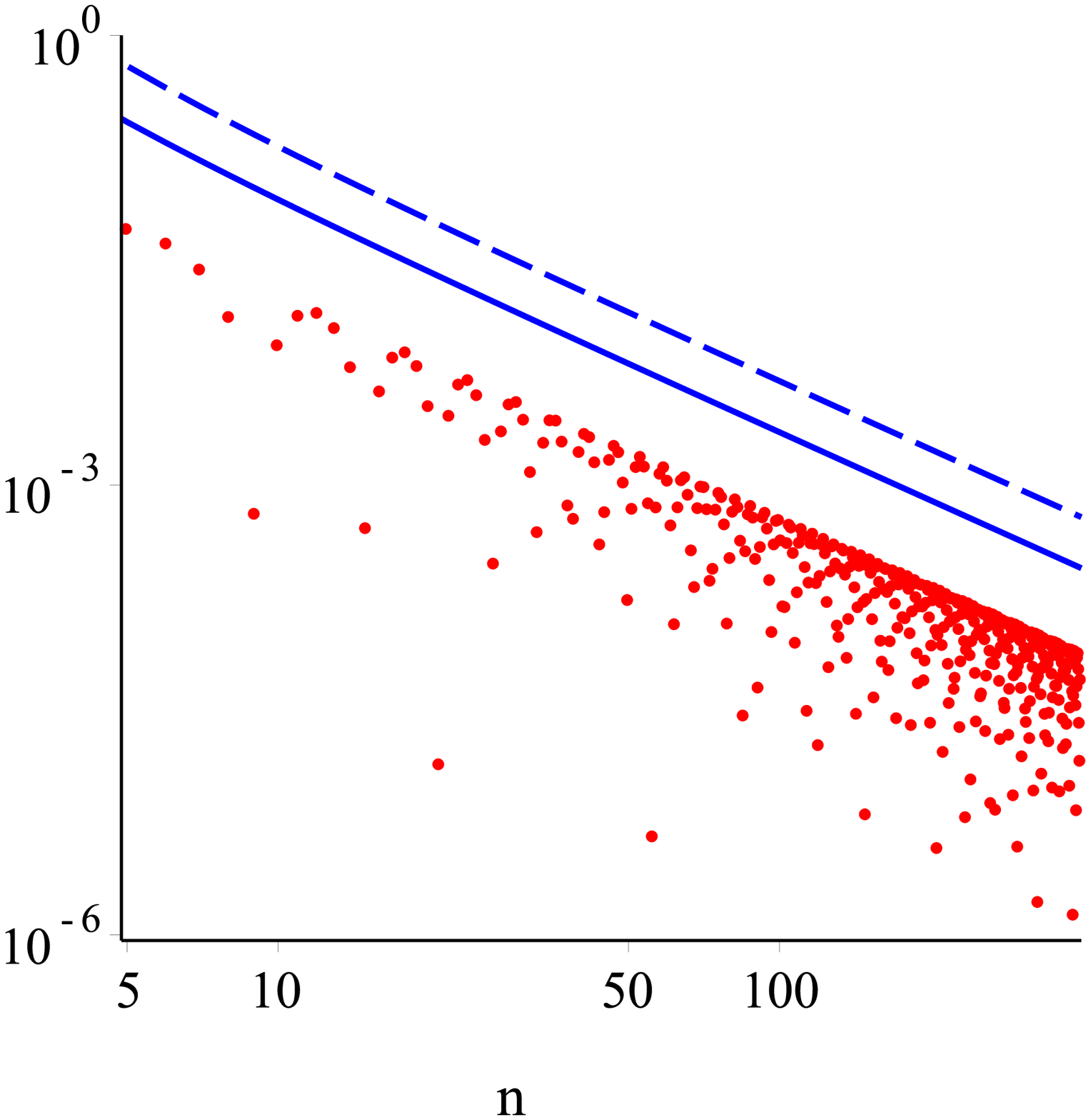}
\end{overpic}
\caption{The bound $B_1(n)$ (line), the bound $B_2(n)$ (dash) and
$|a_n|$ (dots) for $t=0$ (left) and $t=\frac{6}{7}$ (right). Here
$n$ ranges from 5 to 400. } \label{fig:Example}
\end{figure}

We now consider the Legendre polynomial approximation by truncating
the first $N$ terms of \eqref{eq:LegExp}, i.e.,
\begin{align}\label{eq:FiniteLegendre}
f_N(x) = \sum_{n=0}^{N-1} a_n P_n(x).
\end{align}
The following theorem is a corollary of Theorem \ref{thm:NewBound}.
\begin{theorem}\label{thm:LegendreBound}
Under the assumptions of Theorem \ref{thm:NewBound} and assume that
$m\geq1$.
\begin{itemize}
\item When $m=1$, then for each $N\geq 3$,
\begin{align}
\left\|f(x) - f_N(x) \right\|_{\infty} \leq
\frac{4V_1}{\sqrt{\pi(2N-5)}}.
\end{align}
\item When $m\geq2$, then for each $N\geq m+1$,
\begin{align}
\left\|f(x) - f_N(x) \right\|_{\infty} \leq
\frac{2V_m}{(m-1)\sqrt{\pi(2N-2m-1)}} \prod_{k=2}^{m} h_{N-k}.
\end{align}
\end{itemize}
\end{theorem}
\begin{proof}
Recall the well-known inequality $|P_n(x)| \leq 1$ for $x\in[-1,1]$,
we have
\begin{align}\label{eq:MaxBound}
\| f(x) - f_N(x) \|_{\infty} &\leq \sum_{n=N}^{\infty} |a_n|.
\end{align}
We first consider the case $m=1$. By using Theorem
\ref{thm:NewBound}, we obtain
\begin{align}\label{eq:FirstCase}
\| f(x) - f_N(x) \|_{\infty} &\leq \sum_{n=N}^{\infty} \frac{2V_m
h_{n-1}}{\sqrt{\pi(2n-3)}} = \frac{2V_m}{\sqrt{2\pi}}
\sum_{n=N}^{\infty} \frac{1}{(n-\frac{1}{2}) \sqrt{n-\frac{3}{2}}}.
\end{align}
Note that
\begin{align}
\sum_{n=N}^{\infty} \frac{1}{(n-\frac{1}{2}) \sqrt{n-\frac{3}{2}}}
&\leq \sum_{n=N}^{\infty}
\frac{1}{(n-\frac{3}{2})^{\frac{3}{2}}} \nonumber \\
&\leq \int_{N-1}^{\infty} \left(x -
\frac{3}{2}\right)^{-\frac{3}{2}} dx \nonumber \\
&= \frac{2}{\sqrt{N-\frac{5}{2}}}. \nonumber
\end{align}
Substituting this into \eqref{eq:FirstCase} gives the desired
result.

Next, we consider the case $m\geq2$. Combining \eqref{eq:MaxBound}
and Theorem \ref{thm:NewBound}, we obtain
\begin{align}\label{eq:SecondCase}
\| f(x) - f_N(x) \|_{\infty} &\leq \sum_{n=N}^{\infty}
\frac{2V_m}{\sqrt{\pi(2n-2m-1)}} \prod_{k=1}^{m} h_{n-k} \nonumber
\\
&\leq \frac{2V_m}{\sqrt{\pi(2N-2m-1)}}
\sum_{n=N}^{\infty}\prod_{k=1}^{m} h_{n-k}.
\end{align}
Observe that
\begin{align}
\prod_{k=1}^{m} h_{n-k} &= \frac{1}{m-1} \left[ \prod_{k=2}^{m}
h_{n-k} - \prod_{k=1}^{m-1} h_{n-k} \right], \nonumber
\end{align}
which implies
\begin{align}\label{eq:ProdSum}
\sum_{n=N}^{\infty} \prod_{k=1}^{m} h_{n-k} &= \frac{1}{m-1}
\sum_{n=N}^{\infty} \left[ \prod_{k=2}^{m} h_{n-k} - \prod_{k=1}^{m-1} h_{n-k} \right] \nonumber \\
&= \frac{1}{m-1} \prod_{k=2}^{m} h_{N-k}.
\end{align}
Substituting \eqref{eq:ProdSum} into \eqref{eq:SecondCase} gives the
desired result. This completes the proof.
\end{proof}

\begin{remark}
Note that the assumption in \cite[Theorem 2.5]{wang2012legendre}
requires $m>1$. Here we have proved a result for the case $m=1$.
\end{remark}

An interesting question is: What is the error bound of Legendre
approximation to the function $f(x)=|x|$? Note that the analysis of
Chebyshev polynomial approximations to this function has been
discussed comprehensively in \cite[Chapter 7]{trefethen2013atap}.
Here we provide a corresponding result for the truncated Legendre
series.

\begin{corollary}\label{col:AbsFun}
Let $f(x)=|x|$ and let $f_N(x)$ be the truncated Legendre series of
$f(x)$. Then, for each $N\geq3$,
\begin{align}\label{eq:AbsFun}
\|f(x) - f_N(x) \|_{\infty} \leq \frac{8}{\sqrt{\pi(2N-5)}}.
\end{align}
\end{corollary}
\begin{proof}
Note that $m=1$ and $V_1 = 2$ for this function. The bound follows
immediately from Theorem \ref{thm:LegendreBound}.
\end{proof}

\begin{remark}
The result in Corollary \ref{col:AbsFun} is actually overestimated.
In fact, numerical experiments show that $\|f(x) - f_N(x)
\|_{\infty} = O(N^{-1})$ as $N\rightarrow\infty$. However, a
rigorous proof is still open.
\end{remark}


\section{Conclusions}
In this paper, we have presented a new and sharper bound for the
Legendre coefficients of differentiable functions. An illustrative
example is provided to demonstrate the sharpness of our results. We
further apply this result to obtain a new error bound of the
truncated Legendre series in the uniform norm.

Finally, we remark that it is possible to extend the result of
Theorem \ref{thm:NewBound} to a more general case. Indeed, from
\cite[Equation 18.14.7]{olver2010nist} we see that the Gegenbauer
polynomial also satisfies a Bernstein-type inequality, e.g.,
\[
(1-x^2)^{\frac{\lambda}{2}} |C_n^{\lambda}(x)| <
\frac{2^{1-\lambda}}{\Gamma(\lambda)} (n+\lambda)^{\lambda-1}, \quad
n\geq0,
\]
where $x\in[-1,1]$ and $0<\lambda<1$ and $C_n^{\lambda}(x)$ denotes
the Gegenbauer polynomial of degree $n$. Therefore, one can expect
that a sharp bound for the Gegenbauer coefficients of differentiable
functions can be obtained in a similar way.

\section*{Acknowledgement}
This work was supported by the National Natural Science Foundation
of China under grant 11671160.



\begin{thebibliography}{10}



\bibitem{Antonov1981estimate}
V. A. Antonov and K. V. Holsevnikov,
\newblock An estimate of the
remainder in the expansion of the generating function for the
Legendre polynomials  (Generalization and improvement of Bernstein's
inequality),
\newblock {\it Vestnik Leningrad University Mathematics}, 13, 163--166, 1981.




\bibitem{lorch1983altern}
L. Lorch,
\newblock Alternative proof of a sharpened
form of Bernstein's inequality for Legendre polynomials,
\newblock {\it Applicable Analysis}, 14, 237--240, 1983.


\bibitem{olver2010nist}
F. W. J. Olver, D. W. Lozier, R. F. Boisvert and C. W. Clark,
\newblock {\it NIST Handbook of Mathematical Functions},
\newblock Cambridge University Press, 2010.


\bibitem{trefethen2013atap}
L. N. Trefethen,
\newblock {\it Approximation Theory and Approximation Practice},
\newblock SIAM, 2013.


\bibitem{wang2012legendre}
H.~Y. Wang and S. H. Xiang,
\newblock On the convergence rates of Legendre approximation,
\newblock {\it Mathematics of Computation}, 81(278), 861--877, 2012.

\bibitem{wang2016gegenbauer}
H.~Y. Wang,
\newblock On the optimal estimates and comparison of Gegenbauer expansion
coefficients,
\newblock {\it SIAM Journal on Numerical Aanlysis}, 54(3),
1557--1581, 2016.


\bibitem{xiang2012error}
S.~H. Xiang,
\newblock On error bounds for orthogonal polynomial expansions and Gauss-type
quadrature,
\newblock {\it SIAM Journal on Numerical Analysis}, 50(3),
1240--1263, 2012.

\bibitem{zhao2013sharp}
X.~D. Zhao, L.~L. Wang and Z.~Q. Xie,
\newblock Sharp error bounds for Jacobi expansions and Gegenbauer--Gauss quadrature of analytic
functions,
\newblock {\it SIAM Journal on Numerical Analysis}, 519(3),
1443--1469, 2013.




\end{thebibliography}
\end{document}